\documentclass[a4paper,12pt,reqno]{amsart}
\usepackage{a4wide}
\usepackage{amsmath}
\usepackage{amssymb}
\usepackage{amsthm}
\usepackage{latexsym}
\usepackage{graphicx}
\usepackage[english]{babel}
         
\newtheorem{dfn} [subsection]{Definition}
\newtheorem{obs} [subsection]{Remark}

\newtheorem{prop}[subsection]{Proposition}

\newtheorem{teor}[subsection]{Theorem}
\newtheorem{lema}[subsection]{Lemma}
\newtheorem{cor} [subsection]{Corollary}
\newcommand{\Zng}{$\mathbb Z^n$-graded $S$-module}

\def\sdepth{\operatorname{sdepth}}
\def\qdepth{\operatorname{hdepth}}

\def\depth{\operatorname{depth}}

\def\deg{\operatorname{deg}}

\def\reg{\operatorname{reg}}
\def\Gin{\operatorname{Gin}}
\def\PP{\operatorname{P}}
\def\H{\operatorname{H}}

\begin{document}
\selectlanguage{english}
\frenchspacing

\numberwithin{equation}{section}

\title{On the Hilbert depth of the Hilbert function of a finitely generated graded module}
\author[Silviu B\u al\u anescu and Mircea Cimpoea\c s]{Silviu B\u al\u anescu$^1$ and Mircea Cimpoea\c s$^2$}
\date{}



\footnotetext[1]{ \emph{Silviu B\u al\u anescu}, University Politehnica of Bucharest, Faculty of
Applied Sciences, 
Bucharest, 060042, E-mail: silviu.balanescu@stud.fsa.upb.ro}
\footnotetext[2]{ \emph{Mircea Cimpoea\c s}, University Politehnica of Bucharest, Faculty of
Applied Sciences, 
Bucharest, 060042, Romania and Simion Stoilow Institute of Mathematics, Research unit 5, P.O.Box 1-764,
Bucharest 014700, Romania, E-mail: mircea.cimpoeas@upb.ro,\;mircea.cimpoeas@imar.ro}

\begin{abstract}
Let $K$ be a field, $A$ a standard graded $K$-algebra and $M$ a finitely generated graded $A$-module.
Inspired by our previous works, see \cite{lucrare2} and \cite{lucrare7}, we study the invariant called \emph{Hilbert depth} of $h_M$, that is
$$ \qdepth(h_M)=\max\{d\;:\; \sum\limits_{j\leq k} (-1)^{k-j} \binom{d-j}{k-j} h_{M}(j) \geq 0 \text{ for all } k\leq d\}, $$
where $h_M(-)$ is the Hilbert function of $M$, and we prove basic results regard it.

Using the theory of hypergeometric functions, we prove that $\qdepth(h_S)=n$, where $S=K[x_1,\ldots,x_n]$. 
We show that $\qdepth(h_{S/J})=n$, if $J=(f_1,\ldots,f_d)\subset S$ is a complete intersection monomial ideal with
$\deg(f_i)\geq 2$ for all $1\leq i\leq d$. Also, we show that $\qdepth(h_{\overline M})\geq \qdepth(h_M)$
for any finitely generated graded $S$-module $M$, where $\overline M=M\otimes_S S[x_{n+1}]$.
\end{abstract}

\maketitle

\section*{Introduction}

Let $S=K[x_1,\ldots,x_n]$ be the ring of polynomials in $n$ variables over a field $K$.
The Hilbert depth of a finitely graded $S$-module $M$ is the maximal depth of a finitely graded $S$-module $N$ with the same
Hilbert series as $M$; see \cite{uli} for further details. In \cite{lucrare2} we proved a new formula for the Hilbert depth
of a quotient $J/I$ of two squarefree monomial ideals $I\subset J\subset S$. This allowed us, in \cite{lucrare7}, to extend
the definition of Hilbert depth to any (numerical) function $h:\mathbb Z\to\mathbb Z_{\geq 0}$ with the property that $h(j)=0$ for $j\ll 0$. 

More precisely, we set 
$$\qdepth(h):=\max\{d\;:\; \sum_{j\leq k} (-1)^{k-j}\binom{d-j}{k-j}h(j)\geq 0\text{ for all }k\leq d\}.$$
Let $A$ be a standard graded $K$-algebra and $M$ a finitely generated graded $A$-module. Since $h_M(-)$, the Hilbert function of $M$, has the property
$h_M(0)=0$ for $j\ll 0$, it makes sense to consider its Hilbert depth, as was defined above; see also Definition \ref{d2}.


Note that, if $I\subset J\subset S$ are squarefree monomial ideals, then $\qdepth(J/I)$ and $\qdepth(h_{J/I})$ are not the same.
However, there is the following connection between these invariants: If 
$$M(J/I)=(J+(x_1^2,\ldots,x_n^2))/(I+(x_1^2,\ldots,x_n^2)),$$
then $\qdepth(J/I)=\qdepth(h_{M(J/I)})$; see Proposition \ref{qq}.

In Proposition \ref{p22} we prove that
$$k_0 \leq \qdepth(h_M) \leq k_0+\frac{h_1}{h_0},$$ 
where $k_0=k_0(M)=\min\{k\;:\;M_k\neq 0\}$, $h_0=h_M(k_0)$ and $h_1=h_M(k_0+1)$.

In Proposition \ref{p23} we prove that if $M$ is of finite length, then 
$$\qdepth(h_M)\leq k_f(M):=\max\{k\;:\;M_k\neq 0\}.$$
In particular, if $M=S/J$ where $J$ is a graded Artinian ideal, we note in Corollary \ref{c23} that 
$$\qdepth(h_{S/J})\leq \reg(S/J).$$
In Proposition \ref{exact} we show that if $0\to U\to M\to N\to 0$ is a short exact sequence of (nonzero) finitely generated
graded $A$-modules, then 
$$\qdepth(h_M)\geq \min\{\qdepth(h_U),\qdepth(h_N)\}.$$
In Proposition \ref{shif} we prove that 
$$\qdepth(h_{M(m)})=\qdepth(h_M)-m,$$
where $M(m)$ is the $m$-th shift module of $M$.

Using the sign of the hypergeometric function ${}_2F_1(-k,n,-n;-1)$, see Lemma \ref{lemuta}, we prove in Theorem \ref{qdeps} that 
$\qdepth(h_S)=n$. Consequently, in Corollary \ref{free} we prove that if 
$$F=S(a)^{n_1}\oplus S(a-1)^{n_2} \oplus S(a_1)\oplus\cdots\oplus S(a_r),$$
 where $n_1,n_2,a,a_j$ are integers such that $n_1>n_2\geq 0$ and $a\geq a_j+2$ for all $1\leq j\leq r$, then 
$$\qdepth(h_F)=n-a.$$
In Theorem \ref{teo} and Corollary \ref{coro} we prove that if $J=(f_1,\ldots,f_r)\subset S$ is a graded complete 
intersection with $\deg(f_i)\geq 2$ for all $1\leq i\leq r$, where $0\leq r\leq n$, then 
$$\qdepth(h_{S/J})=n.$$
In particular, for $r=0$ we obtain a new proof of the fact that $\qdepth(h_S)=n$.

Finally, in Theorem \ref{ultimap} we show that $\qdepth(h_{\overline M})\geq \qdepth(h_M)$, 
where $\overline S=S[x_{n+1}]$, $M$ is a finitely generated $S$-module and $\overline M=M\otimes_S \overline S$.

\section{Basic properties}

Let $K$ be a field and let 
$$A=\bigoplus_{n\geq 0}A_n,$$
 be a standard graded $K$-algebra, i.e. $A$ is finitely generated, 
$A_0=K$ and $A_1$ generates $A$. 

Let 
$$M=\bigoplus_{k\in\mathbb Z}M_k,$$
 be a nonzero graded finitely generated $A$-module.

Since $M$ is finitely generated, $\dim_K(M_k)<\infty$ for all $k\in\mathbb Z$ and $M_k=0$ for $k\ll 0$.
In particular, there exists $k_0(M)\in\mathbb Z$ such that
$$k_0(M):=\min\{k\;:\;M_k\neq 0\}.$$
We consider the \emph{Hilbert function} of $M$, that is
$$h_M(-):\mathbb Z\to\mathbb Z_{\geq 0},\; h_M(k):=\dim_K M_k,\text{ for all }k\in\mathbb Z.$$
We recall the definition of the Hilbert depth of $h_M$ from \cite{lucrare7}.

Let $d$ be an integer and let
\begin{equation}\label{betamk}
\beta^d_k(h_M):=\begin{cases} \sum\limits_{j=k_0(M)}^k (-1)^{k-j} \binom{d-j}{k-j} h_M(j), & k_0(M)\leq k\leq d \\ 0,& \text{ otherwise}\end{cases}.
\end{equation}
From \eqref{betamk} we deduce that
\begin{equation}\label{alfamk}
h_M(k):=\sum_{j=k_0(M)}^k \binom{d-j}{k-j} \beta^d_j(h_M) \text{ for all } k_0(M)\leq k\leq d.
\end{equation}
With the above notation, we have:

\begin{dfn}\label{d2}
The \emph{Hilbert depth} of $h_M$ is
$$\qdepth(h_M):=\max\{d\in\mathbb Z\;:\;\beta^d_h(h_M)\geq 0\text{ for all }k\leq d\}.$$
\end{dfn}

Note that \eqref{betamk}, \eqref{alfamk} and Definition \ref{d2} hold for any function $h:\mathbb Z\to\mathbb Z_{\geq 0}$
with $h(j)=0$ for $j\ll 0$.

\begin{obs}\label{r22}\rm
 If $\beta^d_k(h_M)\geq 0$ for all $k_0(M)\leq k\leq d$, then, from \cite[Corollary 1.4]{lucrare7},
 it follows that $\beta_k^{d'}(h_M)\geq 0$ for all $d\leq d'$ and $k_0(M)\leq k\leq d'$.
 Also, it is clear that $k_0(M)=k_0(h_M)$.
\end{obs}


Let $0\subset I\subsetneq J\subset S=K[x_1,\ldots,x_n]$ be two squarefree monomial ideals. 
We recall the method of computing Hilbert depth of $J/I$ given in \cite{lucrare2}. 
For $0\leq k\leq n$, we let 
$$\alpha_k(J/I)=|\{u\in S\text{ is a squarefree monomial with }u\in J\setminus I\}|.$$
For all $0\leq d\leq n$ and $0\leq k\leq d$, we consider the integers:
\begin{equation}\label{betak}
\beta_k^d(J/I):=\sum_{j=0}^k (-1)^{k-j} \binom{d-j}{k-j} \alpha_j(J/I).
\end{equation}
We recall the following result:

\begin{teor}(\cite[Theorem 2.4]{lucrare2})\label{d1}
With the above notations, the Hilbert depth of $J/I$ is 
$$\qdepth(J/I):=\max\{d\;:\;\beta_k^d(J/I) \geq 0\text{ for all }0\leq k\leq d\}.$$
\end{teor}

We consider the $S$-module
$$M(J/I):=(J+(x_1^2,\ldots,x_n^2))/(I+(x_1^2,\ldots,x_n^2)).$$
It is easy to see that 
\begin{equation}\label{mij}
M(J/I)=\bigoplus_{\substack{u\in S\text{ squarefree monomial}\\ u \in J\setminus I}}Ku.
\end{equation}
From \eqref{mij} and the definition of $\alpha_k(J/I)$'s 
it follows that
\begin{equation}\label{hak}
\alpha_k(\PP_{J/I})=h_{M(J/I)}(k)\text{ for all }0\leq k\leq n.
\end{equation}
From Theorem \ref{d1}, Definition \ref{d2} and \eqref{hak} we get the following result:

\begin{prop}\label{qq}
With the above notations, we have $$\qdepth(J/I)=\qdepth(h_{M(J/I)}).$$
\end{prop}


In the following, all modules are assumed finitely generated over a standard graded $K$-algebra $A$, unless it is stated otherwise:

\begin{prop}\label{p22}
Let $M$ be a nonzero graded $A$-module, $k_0=k_0(M)$, $h_0=h_M(k_0)$, $h_1:=h_M(k_0+1)$. Then:
$$k_0 \leq \qdepth(h_M) \leq k_0+\frac{h_1}{h_0}.$$
\end{prop}

\begin{proof}
It follows from \cite[Proposition 1.5]{lucrare7}.
\end{proof}

Let $M$ be a nonzero graded $A$-module of finite length, i.e. $\dim_K(M)<\infty$. It follows that
there exists $k_f(M)\geq k_0(M)$ such that 
$$k_f(M):=\max\{k\;:\;M_k\neq 0\}.$$
Note that $k_f(M)\leq k_f(h_M)$. Hence, from \cite[Proposition 1.5]{lucrare7} we conclude that:

\begin{prop}\label{p23}
If $M$ is a nonzero graded $A$-module of finite length, then 
$$\qdepth(h_M)\leq k_f(M).$$
\end{prop}

\begin{cor}\label{c23}
Let $I\subset S=K[x_1,\ldots,x_n]$ be an Artinian homogeneous ideal. 
Then
 $$\qdepth(h_{S/I})\leq \reg(S/I).$$
\end{cor}

\begin{proof}
According to \cite[Theorem 18.4]{peeva}, we have that 
$$\reg(S/I)=\max\{k\;:\;(S/I)_k\neq 0\}.$$
The conclusion follows from Proposition \ref{p23}.
\end{proof}

\begin{obs}\rm
Assume $K$ is a field of characteristic zero and let $I\subset S=K[x_1,\ldots,x_n]$ be a homogeneous ideal. Let
$J:=\Gin(I)$ be the generic initial ideal of $I$ with respect to the reverse lexicographic order. It is well known
that $S/I$ and $S/J$ have the same Hilbert function. 
Moreover, according to Bayer and Stillman \cite{bayer}, we have 
$$\reg(S/I)=\reg(S/J).$$
On the other hand, according to Galligo \cite{galligo}, $J:=\Gin(I)$ is strongly stable, hence from the well known result of
Eliahou and Kervaire \cite{EK}, it follows that
$$\reg(S/I)=\reg(S/J)=\max\{\deg(u)\;:\;u\in G(J)\}-1,$$
where $G(J)$ is the minimal set of monomial generators of $J$.

In conclusion, $\qdepth(h_{S/I})\leq \reg(S/I)$ if and only if $\qdepth(h_{S/J})\leq \reg(S/J)$.

The result from Corollary \ref{c23} cannot be extended in general. For instance, the ideal $J=(x_1^3)\subset S=K[x_1,x_2,x_3]$ is
strongly stable and its regularity is $\reg(S/J)=3-1=2$, while $\qdepth(h_{S/J})=3$.
\end{obs}

\begin{prop}\label{exact}
Let $0\to U\to M\to N\to 0$ be a short exact sequence of (nonzero) graded $A$-modules. Then:
$$\qdepth(h_M)\geq \min\{\qdepth(h_U),\qdepth(h_N)\}.$$
\end{prop}

\begin{proof}
It follows from the fact that $h_{M}(k)=h_U(k)+h_N(k)$ for all $k\in\mathbb Z$
and \cite[Proposition 1.10]{lucrare7}.
\end{proof}

\begin{prop}\label{ciuciu}
Let $M$ be a graded $A$-module and let $r>0$ be an integer.
Then 
$$\qdepth(h_{M^{\oplus r}})=\qdepth(h_M).$$
\end{prop}

\begin{proof}
Since $h_{M^{\oplus r}}(k)=r\cdot h_M(k)$ for all $k\in\mathbb Z$,
the conclusion follows from \cite[Proposition 1.11]{lucrare7}.
\end{proof}

If $M$ is a graded $A$-module and $m$ is an integer, then 
$$M(m) = \bigoplus_{k\in\mathbb Z} M(m)_k = \bigoplus_{k\in\mathbb Z} M_{m+k}$$
is the $m$-th shift module of $M$.

\begin{prop}\label{shif}
Let $M$ be a nonzero graded $A$-module and $m\in\mathbb Z$. Then:
\begin{enumerate}
\item[(1)] $k_0(M(m))=k_0(M)-m$.
\item[(2)] If $\dim_K(M)<\infty$ then $k_f(M(m))=k_f(M)-m$.
\item[(3)] $\qdepth(h_{M(m)})=\qdepth(h_M)-m$.
\end{enumerate}
\end{prop}

\begin{proof}
(1) and (3) Since $h_{M(m)}(k)=h_M(k+m)$ for all $k\in\mathbb Z$,
the conclusion follows from \cite[Proposition 1.12]{lucrare7}.

(2) It is obvious.
\end{proof}

\begin{obs}\rm
Let $h:\mathbb Z\to\mathbb Z_{\geq 0}$ such that $h(j)=0$ for $j\ll 0$. Let $k_0=\min\{j\;:\;h(j)>0\}$ and 
$c=\left\lfloor \frac{h(k_0+1)}{h(k_0)}\right\rfloor$. Let $S:=K[x_1,\ldots,x_n]$ and $\mathbf m=(x_1,\ldots,x_n)$.
We claim that there exists an Artinian $S$-module $M$ such that
\begin{equation}\label{hmj}
h_M(j)=\begin{cases} h(j),& k_0\leq j\leq k_0+c \\ 0,&\text{ otherwise} \end{cases}.
\end{equation}
Indeed, we can take 
$$M:=(S/\mathbf m)(-k_0)^{h(k_0)}\oplus (S/\mathbf m)(-k_0-1)^{h(k_0+1)}\oplus \cdots \oplus
(S/\mathbf m)(-k_0-c)^{h(k_0+c)}.$$
From \eqref{hmj} it is easy to deduce that 
$$\qdepth(h)=\qdepth(h_M).$$ 
Note that $M$ is in fact a graded $K$-vector space
of finite dimension.
\end{obs}

\section{Hilbert depth of the Hilbert series of a free $S$-module}

Let $a\in\mathbb C$ and $j$ a nonnegative integer. We denote $(a)_j=a(a+1)\cdots(a+j-1)$, the \emph{Pochhammer symbol}.
The \emph{hypergeometric function} is
$${}_2F_1(a,b,c;z)=\sum_{j\geq 0}\frac{(a)_j(b)_j}{(c)_j}\cdot \frac{z^j}{j!}.$$
First, we prove the following lemma:

\begin{lema}\label{lemuta}
Let $n\geq 1$ be an integer. Then:
\begin{enumerate}
\item[(1)] ${}_2F_1(0,n,-n;-1)=1$ and ${}_2F_1(-1,n,-n;-1)=0$.
\item[(2)] $(-1)^k{}_2F_1(-k,n,-n;-1)>0$ for any $2\leq k\leq n$.
\end{enumerate}
\end{lema}

\begin{proof}
(1) It is obvious from the definition of the hypergeometric function.

(2) Since ${}_2F_1(-k,n,-n;-1)=\sum\limits_{j=0}^k (-1)^j \binom{k}{j} \frac{(n)_j}{(n-j+1)_j}$, 
    in order to prove (2), it is enough to show that for any $n\geq k\geq 2$ we have that:
\begin{equation}\label{enk}
E(n,k):=  \sum_{j=0}^k (-1)^{k-j} \binom{k}{j} (n)_j (n-k+1)_{k-j}  > 0.
\end{equation}
We consider the functions 
$$f_{1,k},f_{2,k}:(0,2) \to \mathbb R,\;f_{1,k}(x)=\frac{1}{(2-x)^{n}},\;f_{2,k}(x)=\frac{1}{x^{n-k+1}}.$$
By straightforwards computation, for all $0\leq j\leq k$ we have that
\begin{equation}\label{f12k}
f_{1,k}^{(j)}(x)=\frac{(n)_j}{(2-x)^{n+j}}\text{ and }f_{2,k}^{(k-j)}(x)=\frac{(-1)^{k-j}(n-k+1)_{k-j}}{x^{n-j+1}},
\end{equation}
where $f^{(j)}$ denotes the $j$-th derivative of the function $f$. Let 
\begin{equation}\label{fk}
f_k:(0,2)\to\mathbb R,\; f_k(x):=f_{1,k}(x)f_{2,k}(x), \; x\in(0,2).
\end{equation}
From \eqref{enk}, \eqref{f12k}, \eqref{fk} and the chain rule of derivatives, it follows that 
\begin{equation}\label{kkk}
E(n,k) = f_k^{(k)}(1).
\end{equation}
We consider the function 
$$g_k:(-1,1)\to\mathbb R,\; g_k(x)=f_k(1-x)=\frac{1}{(1+x)^n(1-x)^{n-k+1}}=\frac{(1-x)^{k-1}}{(1-x^2)^n}.$$
Since $g_k^{(k)}(x)=(-1)^k f_k^{(k)}(1-x)$, from \eqref{enk} and \eqref{kkk}, in order to complete the proof,
it is enough to prove that
\begin{equation}\label{gk0}
(-1)^k g_k^{(k)}(0) > 0\text{ for all }k\geq 2.
\end{equation}
If $k\geq 2$ and $j\geq 1$, then, using the identity $g_k(x)=(1-x)g_{k-1}(x)$, we deduce that
\begin{equation}\label{gkx}
g_k^{(j)}(x) = (1-x)g_{k-1}^{(j)}(x) - j g_{k-1}^{(j-1)}(x)\text{ for all }x\in(-1,1).
\end{equation}
For $k\geq 1$ and $j\geq 0$ we denote $c_k^{(j)}:=g_k^{(j)}(0)$. Since 
$$g_1(x)=\frac{1}{(1-x^2)^n}=\sum_{\ell=0}^{\infty}\binom{n+\ell-1}{\ell-1}x^{2\ell},$$
it follows that
\begin{equation}\label{pisi1}
c_1^{(j)}=\begin{cases} 0,& j=2\ell+1 \\ \binom{n+\ell-1}{\ell-1}(2\ell)!,&j=2\ell  \end{cases}
\end{equation}
Also, it is clear that
\begin{equation}\label{pisi2}
c_k^{(0)}=1\text{ for all }k\geq 1.
\end{equation}
On the other hand, from \eqref{gkx} it follows that
\begin{equation}\label{pisi3}
c_k^{(j)}=c_{k-1}^{(j)}-j c_{k-1}^{(j-1)}\text{ for all }k\geq 2,\;j\geq 1.
\end{equation}
From \eqref{pisi1}, \eqref{pisi2} and \eqref{pisi3}, using induction on $k\geq 2$, we
can easily deduce that
$$(-1)^jc_k^{(j)}>0\text{ for all }k\geq 2,\;j\geq 0.$$
In particular, it follows that 
$$(-1)^k c_k^{(k)} = (-1)^k g_k^{(k)}(0) > 0 \text{ for all }k\geq 2,$$
hence the proof is complete.
\end{proof}

\begin{teor}\label{qdeps}
Let $S:=K[x_1,\ldots,x_n]$. Then $\qdepth(h_S)=n$.
\end{teor}

\begin{proof}
The Hilbert function of $S$ is $h_S(k)=\binom{n-1+k}{k}$ for all $k\geq k_0(S)=0$. Therefore, from \eqref{betamk}, we have
\begin{equation}\label{bdsk}
\beta_k^d(h_S)=\sum_{j=0}^k (-1)^{k-j} \binom{d-j}{k-j}\binom{n-1+j}{j}\text{ for all }0\leq k\leq d.
\end{equation}
Since $k_0(S)=0$, $h_S(0)=1$ and $h_S(1)=n$, from Proposition \ref{p22} we get that $\qdepth(h_S)\leq n$.
From \eqref{bdsk} we have that
\begin{equation}\label{bd1}
\beta_k^n(h_S)=\sum_{j=0}^k (-1)^{k-j} \binom{n-j}{k-j}\binom{n-1+j}{j}\text{ for all }0\leq k\leq n.
\end{equation}
From \eqref{bd1} if follows that
\begin{equation}\label{bd2}
\beta_k^n(h_S)=(-1)^k \binom{n}{k} {}_2 F_{1} (-k,n,-n;-1) \text{ for all }0\leq k\leq n,
\end{equation}
From \eqref{bd2} and Lemma \ref{lemuta} it follows that $\qdepth(h_S)\geq n$, as required.
\end{proof}

\begin{cor}\label{free}
Let $F=S(a)^{n_1}\oplus S(a-1)^{n_2} \oplus S(a_1)\oplus\cdots\oplus S(a_r)$ where $n_1,n_2,a,a_j$ are some integers such that $n_1>n_2\geq 0$ and 
$a\geq a_j+2$ for all $1\leq j\leq r$. Then $\qdepth(h_F)=n-a$.
\end{cor}

\begin{proof}
From Theorem \ref{qdeps}, Proposition \ref{ciuciu} and Proposition \ref{shif} it follows that 
\begin{align*}
& \qdepth(h_{S(a)^{n_1}})=n-a,\\ 
& \qdepth(h_{S(a-1)^{n_2}})=n-a+1, \text{ if }n_2>0\text{ and }\\
& \qdepth(h_{S(a_j)})=n-a_j\text{ for all }1\leq j\leq r.
\end{align*}
Using Proposition \ref{exact}, we deduce that 
\begin{equation}\label{na}
\qdepth(h_F)\geq \min\{n-a,\;n-a+1,\;n-a_j,\;1\leq j\leq r\}=n-a.
\end{equation}
On the other hand, from hypothesis, we have 
\begin{align*}
& h_F(-a)=\dim_K S(a)^{n_1}_{-a} = n_1 \dim_K S_0 =n_1\text{ and }\\
& h_F(-a+1)=\dim_K S(a)^{n_1}_{-a+1} + \dim_K S(a-1)^{n_2}_{-a+1} = n\cdot n_1 + n_2.
\end{align*}
Since $n_1>n_2$, from Proposition \ref{p22} it follows
that $\qdepth(h_F)\leq n-a$. Hence, the conclusion follows from \eqref{na}.
\end{proof}

\section{Hilbert depth of the Hilbert series of a complete intersection}

Let $S=K[x_1,\ldots,x_n]$ and $J=(f_1,\ldots,f_n)\subset S$ be a graded complete intersection ideal with $d_i=\deg(f_i)\geq 2$ for
all $1\leq i\leq n$. The Hilbert series of $S/J$ is
$$H_{S/J}(t)=\sum_{k\geq 0}h_{S/J}(k)t^k=(1+t+\cdots+t^{d_1-1})(1+t+\cdots+t^{d_2-1})\cdots(1+t+\cdots+t^{d_n-1}).$$

\begin{teor}\label{teo}
With the above notations, we have that 
$$\qdepth(h_{S/J})=n.$$
\end{teor}

\begin{proof}
First, note that $\qdepth(h_{S/J})\leq n$ by Proposition \ref{p22}, since $h_{S/J}(0)=1$ and $h_{S/J}(1)=n$.

We use induction on $n\geq 1$ and $d:=d_1+\cdots+d_n\geq 2n$. If $n=1$ then there is nothing to prove.
If $d=2n$, that is $d_i=2$ for all $1\leq i\leq n$, then:
$$\beta_k^n(h_{S/J})=\sum_{j=0}^k (-1)^{k-j} \binom{n-j}{k-j}\binom{n}{j} = \sum_{j=0}^k (-1)^{k-j} (-1)^{k-j} \binom{n}{k}\binom{k}{j}.$$
Therefore, $\beta_0^n(h_{S/J})=1$ and $\beta_k^n(h_{S/J})=0$ for $2\leq k\leq n$. From Remark \ref{r22}, it follows that $\qdepth(h_{S/J})\geq n$
and thus $\qdepth(h_{S/J})=n$.

Assume $d>2n$. Without any loss of generality, we may assume that $d_n\geq 3$. Let $I=(g_1,\ldots,g_n)$ be a graded complete intersection
ideal with $\deg(g_i)=\deg(f_i)=d_i$ for $1\leq i\leq n-1$ and $\deg(g_n)=d_n-1$. 
Let $J'=(f'_1,\ldots,f'_{n-1})\subset S'=K[x_1,\ldots,x_{n-1}]$
be a graded complete intersection ideal with $\deg(f'_i)=\deg(f_i)=d_i$ for $1\leq i\leq n-1$.
We have that that 
$$H_{S/J}(t)=(1+t+\cdots+t^{d_1-1})\cdots (1+t+\cdots+t^{d_{n-1}-1})(1+t+\cdots+t^{d_n-2}+t^{d_n-1})= $$
\begin{equation}\label{41}
= H_{S/I}(t) + t^{d_n-1}H_{S'/J'}(t).
\end{equation}
From \eqref{41}, it follows that for $0\leq k\leq n$ we have that
\begin{align*}
& \beta_k^n(h_{S/J}) = \sum_{j=0}^k (-1)^{k-j} \binom{n-j}{k-j} h_{S/J}(j) = \\
& = \sum_{j=0}^k (-1)^{k-j} \binom{n-j}{k-j} (h_{S/I}(j)+h_{S'/J'}(j-d_n+1)) = \\
& = \sum_{j=0}^k (-1)^{k-j} \binom{n-j}{k-j} h_{S/I}(j) + \sum_{j=0}^k (-1)^{k-j} \binom{n-j}{k-j}h_{S'/J'}(j-d_n+1) = 
\end{align*}
\begin{equation}\label{42}
= \beta_k^n(h_{S/I}) +  \sum_{j=0}^k (-1)^{k-j} \binom{n-j}{k-j}h_{S'/J'}(j-d_n+1).
\end{equation}
From induction hypothesis, it follows that $\beta_k^n(h_{S/I})\geq 0$ for all $0\leq k\leq n$.
If $k<d_n-1$ then from \eqref{42} it follows that 
\begin{equation}\label{43}
\beta_k^n(h_{S/J})=\beta_k^n(h_{S/I})\geq 0.
\end{equation}
If $k\geq d_n-1$ then 
\begin{align*}
& \sum_{j=0}^k (-1)^{k-j} \binom{n-j}{k-j}h_{S'/J'}(j-d_n+1) = \sum_{j=d_n-1}^k  (-1)^{k-j} \binom{n-j}{k-j}h_{S'/J'}(j-d_n+1) =\\
& = \sum_{j'=0}^{k-d_n+1} (-1)^{(k-d_n+1)-j'} \binom{(n-d_n+1)-j'}{(k-d_n+1)-j'}h_{S'/J'}(j') = \beta^{n-d_n+1}_{k-d_n+1}(h_{S'/J'}).
\end{align*}
Therefore, from \eqref{42} and the induction hypothesis it follows that
\begin{equation}\label{44}
\beta_k^n(h_{S/J})=\beta_k^n(h_{S/I})+\beta_{k-d_n+1}^{n-d_n+1}(h_{S'/J'})\geq 0.
\end{equation}
The conclusion follows from \eqref{43} and \eqref{44}.
\end{proof}

\begin{cor}\label{coro}
If $J=(f_1,\ldots,f_r)\subset S$ is a graded complete intersection with $\deg(f_i)\geq 2$ for all $1\leq i\leq r$, where $0\leq r\leq n$,
then 
$$\qdepth(h_{S/J})=n.$$
In particular, we reobtain the result $\qdepth(h_S)=n$.
\end{cor}

\begin{proof}
Let $\overline J=(g_1,\ldots,g_n)$ be a graded complete intersection with $\deg(g_i)=\deg(f_i)$ for all $1\leq i\leq r$
and $\deg(g_i)=n+1$ for $r+1\leq i\leq n$.
Since 
\begin{align*}
& H_{S/J}(t)=(1+t+\cdots+t^{d_1-1})\cdots (1+t+\cdots+t^{d_r-1})\cdot (1+t+t^2+\cdots )^{n-r}\text{ and }\\
& H_{S/\overline{J}}(t)=(1+t+\cdots+t^{d_1-1})\cdots (1+t+\cdots+t^{d_r-1})\cdot (1+t+t^2+\cdots+t^n)^{n-r},
\end{align*}
it follows that $h_{S/J}(j)=h_{S/\overline{J}}(j)$ for all $0\leq j\leq n$. Therefore 
$$\beta_k^n(h_{S/J})=\beta_k^n(h_{S/\overline{J}})\text{ for all }0\leq k\leq n,$$
hence the result follows from Theorem \ref{teo}.
\end{proof}

\section{Hilbert depth of the Hilbert series of a tensor product of modules}

As in the beginning of the section, $K$ is a field, $A$ is a standard graded $K$-algebra and the modules over $A$
are considered finitely generated and graded unless is stated otherwise. 

We recall the following well known lemma,
regarding the Hilbert series of a tensor product of modules, for which we sketch a proof in order of completion.

\begin{lema}\label{l31}
Let $M,N$ be two $A$-modules such that $N$ is flat. Then:
$$\H_{M\otimes_A N}(t)=\frac{\H_M(t)\H_N(t)}{\H_A(t)}.$$
\end{lema}

\begin{proof}
Take a free resolution of $M$,
\begin{equation}\label{rezo}
\cdots \to F_2\to F_1\to F_0\to M\to 0
\end{equation}
where each $F_n$ is concentrated in degrees $\geq n$. It follows that
\begin{equation}\label{diaree}
\H_M(t)=\sum_{i\geq 0}(-1)^i \H_{F_i}(t).
\end{equation}
Taking $\otimes_A N$ in \eqref{rezo} we get an exact sequence
\begin{equation}\label{exacts}
\cdots \to F_2\otimes_A N \to F_1\otimes_A N \to F_0\otimes_A N \to M\otimes_A N \to 0,
\end{equation}
Since $F_i$ is free, it follows that 
$$\H_{F_i\otimes_A N}(t)=\frac{\H_{F_i}(t)\H_N(t)}{\H_A(t)}\text{ for all }i,$$ 
and thus from \eqref{diaree} we get
$$\H_{M\otimes_A N}(t) = \sum_{i\geq 0}(-1)^i \frac{\H_{F_i}(t)\H_N(t)}{\H_A(t)} = \frac{\H_N(t)}{\H_A(t)}\sum_{i\geq 0}(-1)^i \H_{F_i}(t)
=\frac{\H_M(t)\H_N(t)}{\H_A(t)}, $$
as required.
\end{proof}

\begin{lema}\label{l32}
Let $S=K[x_1,\ldots,x_n]$, $\overline S=S[x_{n+1}]$ and $M$ be a $S$-module. 

If $\overline M=M[x_{n+1}]:=M\otimes_S \overline S$, then
$$\H_{\overline M}(t)=\frac{\H_M(t)}{(1-t)}.$$
In particular, $h_{\overline M}(j)=\sum\limits_{\ell\leq j}h_M(\ell)$.
\end{lema}

\begin{proof}
Since $H_S(t)=\frac{1}{(1-t)^n}$, $H_{\overline S}(t)=\frac{1}{(1-t)^{n+1}}$ and $\overline S$ is flat over $S$, the conclusion follows from Lemma \ref{l31}.
\end{proof}

\begin{teor}\label{ultimap}
Let $S=K[x_1,\ldots,x_n]$, $\overline S=S[x_{n+1}]$, $M$ be a $S$-module and $\overline M=M[x_{n+1}]$. Then
$$\qdepth(h_{\overline M})\geq \qdepth(h_M).$$
\end{teor}

\begin{proof}
Let $d=\qdepth(h_M)$, $k_0=k_0(M)$ and $k_0\leq k\leq d$. By \eqref{betamk} we have that
\begin{equation}
\beta_k^d(h_M)=\sum_{j=k_0}^k (-1)^{k-j} \binom{d-j}{k-j} h_M(j) \geq 0.
\end{equation}
By \eqref{betamk}, Remark \ref{r22} and Lemma \ref{l32} it follows that
\begin{align*}
& \beta_k^d(h_{\overline M})=\sum_{j=k_0}^k (-1)^{k-j} \binom{d-j}{k-j} h_{\overline M}(j) = 
 \sum_{j=k_0}^k (-1)^{k-j} \binom{d-j}{k-j} \sum_{\ell=k_0}^j h_{M}(\ell) = \\
& = \sum_{t=0}^{k-k_0} \sum_{j=k_0}^k (-1)^{k-j} \binom{d-j}{k-j} h_{M}(\ell-t) = (j'=j-t) = \\
& = \sum_{t=0}^{k-k_0} \sum_{j'=k_0}^{k-t} (-1)^{(k-t)-j'} \binom{(d-t)-j'}{(k-t)-j'}h_{M}(j') =
     \sum_{t=0}^{k-k_0} \beta_{k-t}^{d-t}(M)\geq 0,
\end{align*}
as required.
\end{proof}

\begin{obs}\rm
Let $k_0:=k_0(M)$ and let $k_0\leq k\leq d$ be some integers. 
We have that
\begin{align*}
& \beta_k^d(h_{\overline M})= \sum_{j=k_0}^k (-1)^{k-j} \binom{d-j}{k-j} \sum_{\ell=k_0}^j h_M(\ell) = \\
& = \sum_{\ell=k_0}^k (-1)^{k-\ell} \left( \sum_{j=\ell}^k (-1)^{j-\ell} \binom{d-j}{k-j} \right) h_M(\ell) = \\
& = \sum_{\ell=k_0}^k (-1)^{k-\ell} \binom{d-\ell}{k-\ell} {}_2 F_{1}(1,-k+\ell,-d+\ell;-1) h_M(\ell).
\end{align*}
On the other hand, for any noninteger $s$, it holds that
$$ {}_2 F_{1}(1,-s,-s;-1)=\sum_{\ell=0}^s (-1)^{\ell} =\begin{cases} 1,&s\text{ is even }\\0,&s\text{ is odd }\end{cases}.$$
Therefore, we get
$$\beta_d^d(h_{\overline M})=\sum_{\ell=k_0}^k h_M(\ell|d) \geq 0,\text{ where }
h_M(\ell|d) = \begin{cases} h_M(\ell),& d\equiv \ell(\bmod\;2) \\ 0,& \text{ otherwise } \end{cases}.$$
\end{obs}

\subsection*{Aknowledgments} 

The second author, Mircea Cimpoea\c s, was supported by a grant of the Ministry of Research, Innovation and Digitization, CNCS - UEFISCDI, 
project number PN-III-P1-1.1-TE-2021-1633, within PNCDI III.





\end{document}